\documentclass[smallextended]{article}

\usepackage{amsmath}
\usepackage{amsfonts}
\usepackage{amssymb}
\usepackage{mathrsfs}
\usepackage{amscd}
\usepackage{pb-diagram}
\usepackage{amsthm}
\usepackage{color}
\usepackage[all]{xy}
\usepackage{graphicx}
\usepackage{url}
\usepackage{enumerate}
\usepackage[titletoc,title]{appendix}
\usepackage{dsfont}
\usepackage{multirow} 
\numberwithin{equation}{section} 
\usepackage{tikz-cd}
\numberwithin{equation}{section} 
\usepackage{caption}
\DeclareUnicodeCharacter{2212}{-}
\usepackage{tikz}
\usepackage{hyperref}

\usepackage{titlesec}
\titleformat{\subsection}[runin]{\normalsize\bfseries}{\thesubsection}{5pt}{}

\usepackage{tikz}          
\usetikzlibrary{patterns.meta, matrix, arrows, decorations.pathmorphing}
\usetikzlibrary{patterns}

\author{ Danuzia Figueir\^edo \\ \texttt{danuzianf@hotmail.com}
\and     Mathieu Molitor      \\ \texttt{pergame.mathieu@gmail.com} 
}

\date{
\it \small{Instituto de Matem\'{a}tica, Universidade Federal da Bahia}\\
\it \small{Salvador, Brazil}
}

\title{Toric dually flat manifolds and complex space forms}

\begin{document}

\theoremstyle{definition}
\newtheorem{lemma}{Lemma}[section]
\newtheorem{definition}[lemma]{Definition}
\newtheorem{proposition}[lemma]{Proposition}
\newtheorem{corollary}[lemma]{Corollary}
\newtheorem{theorem}[lemma]{Theorem}
\newtheorem{remark}[lemma]{Remark}
\newtheorem{example}[lemma]{Example}
\bibliographystyle{alpha}

\maketitle 


\begin{abstract}
	We classify 1-dimensional connected dually flat manifolds $M$ that are toric in the sense of \cite{molitorToric}, 
	and show that the corresponding torifications are complex space forms. 
	Special emphasis is put on the case where $M$ is an exponential family defined over a finite set. 
\end{abstract}

\section{Introduction}

    A \textit{dually flat manifold} is a triple $(M,h,\nabla)$, where $M$ is a manifold, 
	$h$ is a Riemannian metric on $M$ and $\nabla$ is a flat torsion-free connection 
	on $M$ (not necessarily the Riemannian connection) whose dual connection 
	is also torsion free and flat (see Section \ref{nkwwnkwnknknfkn}). The notion was introduced in a slightly different form
	and studied systematically by H. Shima \cite{shima3,shima} under the name \textit{Hessian manifold}, 
	as a real-analogue of K\"{a}hler manifold. In the context of Information Geometry \cite{Jost2,Murray}, 
	dually flat manifolds attracted much attention since S. Amari and H. Nagaoka \cite{Amari-Nagaoka} showed
	that every exponential family $\mathcal{E}$ equipped with the Fisher metric $h_{F}$ and exponential connection 
	$\nabla^{(e)}$ is naturally a dually flat manifold. 

	More recently, the second author investigated the relationship between toric geometry 
	and dually flat manifolds through a geometric construction that we call \textit{torification} \cite{molitorToric}.
	This construction associates to a given dually flat manifold $(M,h,\nabla)$ satisfying certain properties, 
	a K\"{a}hler manifold $N$ equipped with an effective holomorphic and isometric torus action $\Phi:\mathbb{T}^{n}\times N\to N$, 
	where $\mathbb{T}^{n}=\mathbb{R}^{n}/\mathbb{Z}^{n}$ is the $n$-dimensional real torus and $
	n=\textup{dim}_{\mathbb{R}}(M)=\textup{dim}_{\mathbb{C}}(N)$. The manifold $N$ is then called 
	the \textit{torification} of $M$. For instance, 
	the set $\mathcal{B}(n)$ of all binomial distributions 
	$p(k)=\binom{n}{k}q^{k}(1-q)^{n-k}$, $q\in (0,1)$,
	defined over $\Omega=\{0,1,...,n\}$, is a 1-dimensional exponential family whose torification is 
	the 1-dimensional complex projective space $\mathbb{P}_{1}(c)$ of holomorphic sectional curvature $c=1/n$, endowed with the 
	torus action $e^{i\theta}\cdot [z_{0},z_{1}]=[e^{i\theta}z_{0},z_{1}]$ (homogeneous coordinates). 
	The construction is motivated by the geometrization program of Quantum Mechanics pursued by the second author 
	in previous works \cite{Molitor2012, Molitor-exponential,Molitor2014,Molitor-2015,molitorToric,molitor2023geometric}.
	This approach, although still in development, has already shed new light on some well-established quantum problems. 
	For example, it is proven that the mathematical description of the spin of a non-relativistic quantum particle can be entirely 
	recovered from the family of binomial distributions\footnote{To give a rough idea on how this works, 
	let us mention the following fact: the torification $N$ of an exponential family $\mathcal{E}$ 
	defined over a finite sample space $\Omega=\{x_{0},...,x_{n}\}$ comes automatically with a K\"{a}hler 
	immersion $f:N\to \mathbb{P}_{n}(1)$ \cite{molitorToric}. When $\mathcal{E}$ is the family of binomial distributions, 
	the map $f$ is the Veronese embedding, and the latter is essentially the spin coherent state map 
	associated to the group $SU(2)$ \cite{Eva}. From this, one can apply appropriate 
	geometric/analytical constructions to recover the correct mathematical description of the quantum Hilbert space. 
	For example, one can apply Odzijewicz's approach of Quantum Mechanics \cite{Anatol}.}.
	On the mathematical side, applications in toric and information geometry 
	are discussed in \cite{molitorToric,molitor2023geometric,molitor-weyl} and \cite{fujita}. 

	In this paper, we classify 1-dimensional connected dually flat manifolds $(M,h,\nabla)$ 
	that have a torification $N$ which satisfies the following conditions: (1) $M$ has a global affine coodinate system $x:M\to \mathbb{R}$, 
	(2) $N$ is \textit{regular} (see below) and (3) the space of K\"{a}hler functions on $N$ separates the points of $N$ 
	(see Section \ref{nfeknkfnekndk}). Here, the word ``regular" refers to the requirement that $N$ is connected, simply-connected, complete and 
	that the K\"{a}hler metric on $N$ is real analytic. This is an important condition because it guarantees 
	uniqueness of the torification up to equivariant K\"{a}hler isomorphisms 
	(see Proposition \ref{newdnkekfwndknk}). If a dually flat manifold 
	has a regular torification, we say that it is \textit{toric}. In this paper, we are mostly interested in 
	toric dually flat spaces and their regular torifications. 

	Let us describe our results. Let $F(c)$ denote the complete and simply connected complex space form of complex dimension 1 
	and constant holomorphic sectional curvature $c\in \mathbb{R}$. It is easy to define, 
	for each $c\in \mathbb{R}$, a toric dually flat manifold 
	$M_{c}$ whose regular torification is $F(c)$. Our main result (Theorem \ref{nefknknfeknkn}) says that 
	a 1-dimensional connected dually flat manifold $(M,h,\nabla)$ has a torification $N$ that satisfies (1), (2) and (3)
	if and only if $M\cong M_{c}$ (isomorphism of dually flat manifolds) for some $c\in \mathbb{R}$, 
	in which case $N$ is equivariantly K\"{a}hler isomorphic to $F(c)$. The proof uses the notion of 
	\textit{Hessian sectional curvature} \cite{shima} and a classification result proven in \cite{Molitor-Hessian}.

	An interesting ``quantization phenomenom" occurs when, in addition to the conditions (1), (2) and (3) above, it is assumed that 
	$M=\mathcal{E}$ is an exponential family defined over a finite sample space. In this case, we prove (Theorem \ref{nfeknkfnedkwnk}) 
	that $c=\tfrac{1}{p}$ for some integer $p\geq 1$, and that $M$ is ``essentially" a family of binomial 
	distributions defined over $\{0,1,...,p\}$. Thus, in particular, the regular torification is 
	the complex projective space $\mathbb{P}_{1}(c)$ of holomorphic sectional curvature $c=\tfrac{1}{p}$. \\


	
	The paper is organized as follows. In Section \ref{nknkneknknfwknk}, we give a fairly detailed discussion on 
	the concept of torification, with which not all readers may be familiar. This includes: Dombrowski's construction, 
	parallel and fundamental lattices, lifting procedure, examples from Information Geometry and Hessian sectional curvature. 
	The material is mostly taken from \cite{molitorToric,Molitor-Hessian}. 
	In Section \ref{nceknwkenfknk}, we show that every complete and simply connected complex space form of dimension 1 
	can be realized as a torification of an appropriate dually flat manifold (Proposition \ref{nekwnkendkwnkn}). 
	In Section \ref{nrfekwnkefnekn}, we prove a technical result (Proposition \ref{nfeknkefnknwk}). Our main results (Theorems 
	\ref{nefknknfeknkn} and \ref{nfeknkfnedkwnk}) are proven in Section \ref{neknkrfnekwnk}. \\

	\textbf{Notation.} The derivative of a smooth map $f:M\to M'$ between manifolds 
	is denoted by $f_{*}:TM\to TM'$. If $p\in M$, we write $f_{*_{p}}:T_{p}M\to T_{f(p)}M'$. 
	Let $G$ be a Lie group and $\Phi:G\times M\to M$ a Lie group action. Given $g\in G$, the map $\Phi_{g}$ is the diffeomorphism of $M$ 
	defined by $\Phi_{g}(p)=\Phi(g,p)$, where $p\in M$. 
	A smooth map $f:N\to N'$ between K\"{a}hler manifolds is a \textit{K\"{a}hler immersion} (resp. \textit{K\"{a}hler isomorphism}) 
	if $f$ is an isometric and holomorphic immersion (resp. isometric and holomorphic diffeomorphism).

\section{Torification of dually flat manifolds}\label{nknkneknknfwknk}
	In this section, we discuss the concept of torification, which is used throughout this paper. This concept 
	is a combination of two ingredients: (1) \textit{Dombrowski's construction}, 
	which implies that the tangent bundle of a dually flat manifold is naturally a K\"{a}hler manifold \cite{Dombrowski}, 
	and (2) \textit{parallel lattices}, which are used to implement torus actions. 
	Examples from Information Geometry are presented. The material is taken from \cite{molitorToric,Molitor-Hessian}.

\subsection{Dombrowski's construction.}\label{nkwwnkwnknknfkn}
	Let $M$	be a connected manifold of dimension $n$, endowed with a Riemannian metric $h$ 
	and affine connection $\nabla$ ($\nabla$ is not necessarily 
	the Levi-Civita connection). The \textit{dual connection} of $\nabla$, denoted by $\nabla^{*}$, is the only connection satisfying 
	$X(h(Y,Z))=h(\nabla_{X}Y,Z)+h(Y,\nabla^{*}_{X}Z)$ for all vector fields $X,Y,Z$ on $M$. 
	When both $\nabla$ and $\nabla^{*}$ are flat (i.e., the curvature tensor and torsion are zero), 
	we say that the triple $(M,h,\nabla)$ is a \textit{dually flat manifold}. 

	Let $\pi:TM\to M$ denote the canonical projection. Given a local coordinate system $(x_{1},...,x_{n})$ 
	on $U\subseteq M$, we can define a coordinate system $(q,r)=(q_{1},...,q_{n},r_{1},...,r_{n})$ 
	on $\pi^{-1}(U)\subseteq TM$ by letting $(q,r)(\sum_{j=1}^{n}a_{j}\tfrac{\partial}{\partial x_{j}}\big\vert_{p})=
	(x_{1}(p),...,x_{n}(p),a_{1},...,a_{n})$, where $p\in M$ and $a_{1},...,a_{n}\in \mathbb{R}$. 
	Write $(z_{1},...,z_{n})=(q_{1}+ir_{1},...,q_{n}+ir_{n})$, where $i=\sqrt{-1}$. When $\nabla$ is flat, Dombrowski 
	\cite{Dombrowski} showed that the family of complex coordinate systems $(z_{1},...,z_{n})$ 
	on $TM$ (obtained from affine coordinates on $M$) form a holomorphic atlas on $TM$. Thus, when $\nabla$ is flat, 
	$TM$ is naturally a complex manifold. If in addition $\nabla^{*}$ is flat, then $TM$ has a natural 
	K\"{a}hler metric $g$ whose local expression in the coordinates $(q,r)$ is given by $g(q,r)=
	\big[\begin{smallmatrix}
		h(x)  &  0\\
		0     &   h(x)
	\end{smallmatrix}
	\big]$, where $h(x)$ is the matrix representation of $h$ in the affine coordinates $x=(x_{1},...,x_{n})$. 
	It follows that the tangent bundle of a dually flat manifold is naturally a K\"{a}hler manifold. In this paper, 
	we will refer to this K\"{a}hler structure as the \textit{K\"{a}hler structure associated to Dombrowski's 
	construction}.

\subsection{Parallel lattices.}\label{nekwnkefkwnkk}

	Let $(M,h,\nabla)$ be a dually flat manifold of dimension $n$. 
	A subset $L\subset TM$ is said to be a \textit{parallel lattice} with respect to $\nabla$ if there are $n$ vector 
	fields $X_{1},...,X_{n}$ on $M$ that are parallel with respect to $\nabla$ and such that: $(i)$ $\{X_{1}(p),...,X_{n}(p)\}$ is a basis 
	for $T_{p}M$ for every $p\in M$, and $(ii)$ $L=\{k_{1}X_{1}(p)+...+k_{n}X_{n}(p)\,\,\vert\,\,k_{1},...,k_{n}\in \mathbb{Z},\,\,p\in M\}$. 
	In this case, we say that the frame $X=(X_{1},...,X_{n})$ is a 
	\textit{generator} for $L$. We will denote the set of generators for $L$ by 
	$\textup{gen}(L)$. 

	Given a parallel lattice $L\subset TM$ with respect to $\nabla$, and $X\in \textup{gen}(L)$ , 
	we will denote by $\Gamma(L)$ the set of 
	transformations of $TM$ of the form $u\mapsto u+k_{1}X_{1}+...+k_{n}X_{n}$, where 
	$u\in TM$ and $k_{1},...,k_{n}\in \mathbb{Z}$. The group
	$\Gamma(L)$ is independent of the choice of $X\in \textup{gen}(L)$ and is isomorphic to $\mathbb{Z}^{n}$. 
	Moreover, the natural action of $\Gamma(L)$ 
	on $TM$ is free and proper. Thus the quotient $M_{L}=TM/\Gamma(L)$ is a 
	smooth manifold and the quotient map $q_{L}:TM\to M_{L}$ is a covering map 
	whose Deck transformation group is $\Gamma(L)$. Since 
	$\pi\circ \gamma=\pi$ for all $\gamma\in \Gamma(L)$, 
	there exists a surjective submersion $\pi_{L}:M_{L}\to M$ such that $\pi=\pi_{L}\circ q_{L}$.

	Let $\mathbb{T}^{n}=\mathbb{R}^{n}/\mathbb{Z}^{n}$ be the $n$-dimensional torus. 
	Given $t=(t_{1},...,t_{n})\in \mathbb{R}^{n}$, we will denote 
	by $[t]=[t_{1},...,t_{n}]$ the corresponding equivalence class in $\mathbb{R}^{n}/\mathbb{Z}^{n}$. 
	Given $X\in \textup{gen}(L)$, we will denote by 
	$\Phi_{X}:\mathbb{T}^{n}\times M_{L}\to M_{L}$ the effective torus action defined by 
	\begin{eqnarray}\label{nkdnknewknek}
		\Phi_{X}([t],q_{L}(u))=q_{L}(u+t_{1}X_{1}+...+t_{n}X_{n}). 
	\end{eqnarray}

	The manifold $M_{L}=TM/\Gamma(L)$ is naturally a K\"{a}hler manifold (this follows from the fact that 
	every $\gamma\in \Gamma(L)$ is a holomorphic and isometric map with respect the K\"{a}hler structure 
	associated to Dombrowski's construction). Moreover, for each $a\in \mathbb{T}^{n}$, the map 
	$(\Phi_{X})_{a}:M_{L}\to M_{L}$, $p\mapsto \Phi_{X}(a,p)$ is a holomorphic isometry. 

\subsection{Torification.}\label{nfkenkfejdefdknkn} 
	Let $(M,h,\nabla)$ be a connected dually flat manifold of dimension $n$ and $N$ a connected K\"{a}hler 
	manifold of complex dimension $n$, equipped with an effective holomorphic and isometric torus 
	action $\Phi:\mathbb{T}^{n}\times N\to N$. Let $N^{\circ}$ denote the set of points $p\in N$ 
	where $\Phi$ is free. 

\begin{definition}[\textbf{Torification}]\label{nfeknkefenkwnk}
	We shall say that $N$ is a \textit{torification} of $M$ 
	if there exist a parallel lattice $L\subset TM$ with respect to 
	$\nabla$, $X\in \textup{gen}(L)$ and a holomorphic and isometric diffeomorphism $F:M_{L}\to N^{\circ}$
	satisfying $F\circ (\Phi_{X})_{a}=\Phi_{a}\circ F$ for all $a\in \mathbb{T}^{n}$. 
\end{definition}

	By abuse of language, we will often say that the torus action $\Phi:\mathbb{T}^{n}\times N\to N$ is a torification of $M$. 
	
	In practice, it is often convenient to work with the following equivalent definition, which is 
	expressed in terms of covering maps. 

\begin{proposition}\label{nkwdnkkfenknk}
	Let $M$ and $N$ be as defined in the beginning of this section. The following are equivalent:   
	\begin{enumerate}[(1)]
	\item $N$ is a torification of $M$. 
	\item There exist a holomorphic and isometric covering map $\tau: TM\to N^{\circ}$, a parallel lattice $L\subset TM$ 
		with respect to $\nabla$ and $X=(X_{1},...,X_{n})\in \textup{gen}(L)$ such that:
			\begin{enumerate}[(i)]
			\item $\Gamma(L)=\textup{Deck}(\tau)$ (= Deck transformation group of $\tau$).
			\item $\tau\circ T_{t}	= \Phi_{[t]}\circ \tau$
				for every $t\in \mathbb{R}^{n}$, where $T:\mathbb{R}^{n}\times TM\to TM$ is the Lie group 
				action of $\mathbb{R}^{n}$ on $TM$ given by $T(t,u)=u+t_{1}X_{1}+...+t_{n}X_{n},$
				where $t=(t_{1},...,t_{n})\in \mathbb{R}^{n}$ and $u\in TM$. 
			\end{enumerate}
	\end{enumerate}
\end{proposition}

	A smooth map $f:M\to M'$ between two dually flat manifolds $(M,h,\nabla)$ and $(M',h',\nabla')$ is said to be 
	an \textit{affine isometric map} if it is affine with respect to $\nabla$ and $\nabla'$, and 
	if $f^{*}h'=h$. An affine isometric map which is also a diffeomorphism is called 
	an \textit{isomorphism of dually flat maniolds}. 
	If there exists an isomorphism of dually flat manifolds between $M$ and $M'$, and if $\Phi:\mathbb{T}^{n}\times N\to N$ 
	is a torification of $M$, then it is also a torification of $M'$.  

	We now discuss uniqueness. We shall say that a K\"{a}hler manifold $N$ is \textit{regular} if it is connected, 
	simply connected, complete and if the K\"{a}hler metric is real analytic. 
	A torification $\Phi:\mathbb{T}^{n}\times N\to N$ is said to be \textit{regular} if $N$ is regular. 
	In this paper, we are mostly interested 
	in regular torifications. 

	Two torifications $\Phi:\mathbb{T}^{n}\times N\to N$ and $\Phi':\mathbb{T}^{n}\times N'\to N'$ 
	of the same connected dually flat manifold $(M,h,\nabla)$ are said to be \textit{equivalent} if there exists 
	a K\"{a}hler isomorphism $f:N\to N'$ and a Lie group isomorphism $\rho:\mathbb{T}^{n}\to \mathbb{T}^{n}$ such that 
	$f\circ \Phi_{a}=\Phi'_{\rho(a)}\circ f$ for all $a\in \mathbb{T}^{n}$. 

\begin{theorem}[\textbf{Equivalence of regular torifications}]\label{newdnkekfwndknk}
	Regular torifications of a connected dually flat manifold $(M,h,\nabla)$ are equivalent. 
\end{theorem}

	A connected dually flat manifold $(M,h,\nabla)$ is said to be \textit{toric} if it has a regular torification $N$. 
	In this case, we will often refer to $N$ as ``the regular torification of $M$", and keep in mind that it is 
	only defined up to an equivariant K\"{a}hler isomorphism and reparametrization of the torus.

\subsection{Fundamental lattices and K\"{a}hler functions.}\label{nfeknkfnekndk}

	If a connected dually flat manifold $(M,h,\nabla)$ is toric, then by definition there are a regular torification $N$, a 
	parallel lattice $L\subset TM$, a generator $X\in \textup{gen}(L)$ and a K\"{a}hler isomorphism $F:M_{L}\to N^{\circ}$ 
	that is equivariant in the sense of Definition \ref{nfeknkefenkwnk}. The quadruple $(L,X,F,N)$ is not unique in general, 
	but the parallel lattice $L$ is; we call it the \textit{fundamental lattice} of $(M,h,\nabla)$ (\cite{molitorToric}, Lemma 10.1 and 
	Definition 10.2).

	The fundamental lattice is closely related to the notion of \textit{K\"{a}hler function}, that we now recall.

\begin{definition}\label{fdkjfekgjrkgj}
	Let $N$ be a K\"{a}hler manifold with K\"{a}hler structure $(g,J,\omega).$
	A smooth function $f:N\rightarrow \mathbb{R}$ is said to be \textit{K\"{a}hler} if it satisfies $\mathscr{L}_{X_{f}}g=0$,
	where $X_{f}$ is the Hamiltonian vector field associated to $f$ 
	(i.e.\ $\omega(X_{f},\,.\,)=df(.)$) and where $\mathscr{L}_{X_{f}}$ is the Lie derivative in the direction of $X_{f}.$
\end{definition}

	Following \cite{Cirelli-Quantum}, we denote by $\mathscr{K}(N)$ the space of K\"{a}hler functions on $N.$ 
	When $N$ is connected, the dimension of $\mathscr{K}(N)$ is at most $\tfrac{1}{2}n(n+1)+1,$ where 
	$n=\textup{dim}_{\mathbb{R}}(N)$. This follows from the fact that the Hamiltonian vector field of a K\"{a}hler function 
	is a Killing vector field, and the space of Killing vector fields on a connected Riemannian manifold $M$ is a most 
	$\tfrac{1}{2}n(n+1)$, where $n=\textup{dim}(M)$ (see \cite{Kobayashi-Nomizu}, Chapter VI, Theorem 3.3).
	If $N$ has a finite number of connected components, then $\mathscr{K}(N)$ is a finite dimensional Lie algebra 
	for the Poisson bracket $\{f,g\}:=\omega(X_{f},X_{g}).$ 
	
	We say that $\mathscr{K}(N)$ \textit{separates the points of $N$} 
	if $f(p)=f(q)$ for all $f\in \mathscr{K}(N)$ implies $p=q$. As a matter of notation, 
	we use $\textup{Diff}(M)$ to denote the group of diffeomorphisms of a manifold $M$.

\begin{proposition}[\cite{molitorToric}, Proposition 10.4]\label{nfeknkwneknk}
	Let $\Phi:\mathbb{T}^{n}\times N\to N$ be a regular torification of the dually flat manifold $(M,h,\nabla)$, with fundamental lattice 
	$\mathscr{L}\subset TM$. Then 
	\begin{eqnarray}\label{njenfnenkfkenk}
		\Gamma(\mathscr{L})\subseteq \big\{\varphi\in \textup{Diff}(TM)\,\,\big\vert\,\, 
		f\circ \varphi=f\,\,\forall\,f\in \mathscr{K}(TM)\big\}. 
	\end{eqnarray}
	When $\mathscr{K}(N)$ separates the points of $N$, \eqref{njenfnenkfkenk} is an equality. (See Section \ref{nekwnkefkwnkk} 
	for the definition of $\Gamma(\mathscr{L})$.)
\end{proposition}

	An immediate consequence is the following 

\begin{corollary}\label{nekdnwkwndknk}
	Let $(M,h,\nabla)$ be a connected dually flat manifold. If $\mathscr{K}(TM)$ separates the points of $TM$, 
	then $M$ is not toric. 
\end{corollary}

	For later reference, we give the following technical result, whose proof can be found in \cite{Molitor2014}. 

\begin{proposition}\label{c kdlfldfkdl}
	Let $(M,h,\nabla)$ be a dually flat manifold and let $(g,J,\omega)$ be the K\"{a}hler structure 
	on $TM$ associated to $(h,\nabla)$ via Dombrowski's construction. Let $(x_{1},...,x_{n})$ be an affine coordinate system 
	with respect to $\nabla$ on $U\subseteq M$, and let $(q,r)=(q_{1},...,q_{n},r_{1},...,r_{n})$ denote the corresponding 
	coordinate system on $\pi^{-1}(U)\subseteq TM$ (see Section \ref{nkwwnkwnknknfkn}). 
	Given a smooth function $f\,:\,TM \rightarrow \mathbb{R}$, 
	we have the following equivalence: $f$ is K\"{a}hler on $\pi^{-1}(U)$ if and only if 
	\begin{eqnarray}\label{efkekfjkdf}
		\left \lbrace 
		\begin{array}{ccc}\label{dewngkkrtjkkre}
			\dfrac{\partial^{2}f}{\partial q_{i}\partial q_{j}}-
			\dfrac{\partial^{2}f}{\partial r_{i}\partial r_{j}} &=& 
			2\displaystyle\sum_{k=1}^{n}\Gamma_{ij}^{k}\circ \pi\,\dfrac{\partial f}{\partial q_{k}},\\
			\dfrac{\partial^{2}f}{\partial q_{i}\partial r_{j}}+
			\dfrac{\partial^{2}f}{\partial q_{j}\partial r_{i}} &=& 
			2\displaystyle\sum_{k=1}^{n}\Gamma_{ij}^{k}\circ \pi\,\dfrac{\partial f}{\partial r_{k}},
		\end{array}
		\right.
	\end{eqnarray}
	for all $i,j=1,...,n,$ where $\Gamma_{ij}^{k}$ are the Christoffel symbols of $h$ in the coordinates $(x_{1},...,x_{n})$. 
\end{proposition}

\subsection{Lifting procedure.}\label{ncekndkrnkenknk} In this section, we discuss briefly a way of lifting affine 
	isometric maps between toric dually flat manifolds to K\"{a}hler immersions 
	between the corresponding torifications (for the proofs, see \cite{molitorToric}, Section 9). 

	Let $(M,h,\nabla)$ be a connected dually flat manifold that has a torification $\Phi:\mathbb{T}^{n}\times N\to N$. 
	By definition, there are a parallel lattice $L\subset TM$, generated by a frame $X$, and an equivariant 
	K\"{a}hler isomorphism $F:TM/\Gamma(L)\to N^{\circ}$, where $\Gamma(L)$ is the group of translations associated to $L$ and 
	$N^{\circ}$ is the set of points $p\in N$ where the torus action $\Phi$ is free 
	(see Section \ref{nfkenkfejdefdknkn}). 

	Let $\pi:TM\to M$ be the canonical projection, and let $q_{L}:TM\to TM/\Gamma(L)$ be the quotient map. 
	The fact that $\pi\circ \gamma=\pi$ for all $\gamma\in \Gamma(L)$ implies 
	that there exists a surjective submersion $\pi_{L}:TM/\Gamma(L)\to M$ such that $\pi=\pi_{L}\circ q_{L}$.
	By inserting $F$, we obtain a factorization $\pi=\kappa\circ \tau$, where 
	$\kappa=\pi_{L}\circ F^{-1}$ and $\tau=F\circ q_{L}$. We say that the formula $\pi=\kappa\circ \tau$ is a 
	\textit{toric factorization}, and call $\tau:TM\to N^{\circ}$ (resp. $\kappa:N^{\circ}\to M$) a \textit{compatible covering map}
	(resp. \textit{compatible projection}). It can be shown that $\tau$ is a covering map whose Deck transformation group is 
	$\Gamma(L)$, and that $\kappa$ is a principal $\mathbb{T}^{n}$-bundle over $M$. 

\begin{proposition}\label{ndcknkdnfknnkk}
	Let $(M,h,\nabla)$ and $(M',h',\nabla')$ be two connected dually flat manifolds with 
	torifications $\Phi:\mathbb{T}^{n}\times N\to N$ and $\Phi':\mathbb{T}^{m}\times N'\to N'$, 
	respectively. Let $\pi=\kappa\circ \tau:TM\to M$ and $\pi'=\kappa'\circ \tau':TM'\to M'$ 
	be toric factorizations. If $N$ and $N'$ are regular, then for every affine isometric map $f:M\to M'$, 
	there is a unique K\"ahler immersion $m:N\to N'$ that makes the following diagram 
	commutative:
\begin{eqnarray}\label{knkdfnknkndknk}
	\begin{tikzcd}
		TM \arrow{rrr}{\displaystyle f_{*}} \arrow[dd,swap,"\displaystyle\pi"] \arrow{rd}{\displaystyle \tau} &    &   & 
			TM' \arrow[swap]{ld}{\displaystyle \tau'} \arrow[dd,"\displaystyle\pi'"]  \\
		& N^{\circ} \arrow{r}{\displaystyle m} \arrow{ld}{\displaystyle \kappa}  
			& (N')^{\circ} \arrow[swap]{dr}{\displaystyle \kappa'}  & \\
		M \arrow[swap]{rrr}{\displaystyle f}  &     && M'
	\end{tikzcd}
\end{eqnarray}
	Moreover, $m$ is equivariant in the sense that there is a Lie group homomorphism $\rho:\mathbb{T}^{n}\to \mathbb{T}^{m}$, 
	with finite kernel, such that $m\circ \Phi_{a}=\Phi'_{\rho(a)}\circ m$ for all $a\in \mathbb{T}^{n}$. 
\end{proposition}

	We call $m$ a \textit{lift} of $f$.

\subsection{Examples from Information Geometry.}\label{nknkefnknk} 

	The concept of torification was motivated, in the first place, by the 
	connection between K\"{a}hler geometry, Information Geometry and Quantum Mechanics.
	In this section, we illustrate this connection with a few examples. The reader interested in 
	Information Geometry may consult \cite{Jost2,Amari-Nagaoka,Murray}.

\begin{definition}\label{def:5.1}
	A \textit{statistical manifold} is a pair $(S,j)$, where $S$ is a manifold and where $j$ 
	is an injective map from $S$ to the space of all probability density functions $p$ 
	defined on a fixed measure space $(\Omega,dx)$: 
	\begin{eqnarray*}
		j:S\hookrightarrow\Bigl\{ p:\Omega\to\mathbb{R}\;\Bigl|\; p\:\textup{is measurable, }p\geq 
		0\textup{ and }\int_{\Omega}p(x)dx=1\Bigr\}.
	\end{eqnarray*}
\end{definition}

	If $\xi=(\xi_{1},...,\xi_{n})$ is a coordinate system on a statistical manifold $S$, then we shall 
	indistinctly write $p(x;\xi)$ or $p_{\xi}(x)$ for the probability density function determined by $\xi$.

	Given a ``reasonable" statistical manifold $S$, it is possible to define a metric $h_{F}$ and a 
	family of connections $\nabla^{(\alpha)}$ on $S$ $(\alpha\in\mathbb{R})$ in the following way: 
	for a chart $\xi=(\xi_{1},...,\xi_{n})$ of $S$, define
	\begin{alignat*}{1}
		\bigl(h_F\bigr)_{\xi}\bigl(\partial_{i},\partial_{j}\bigr) & :=
		\mathbb{E}_{p_{\xi}}\bigl(\partial_{i}\ln\bigl(p_{\xi}\bigr)\cdotp\partial_{j}\ln\bigl(p_{\xi}\bigr)\bigr),
		\nonumber\\
		\Gamma_{ij,k}^{(\alpha)}\bigl(\xi\bigr) & :=
		\mathbb{E}_{p_{\xi}}\bigl[\bigl(\partial_{i}\partial_{j}\ln\bigl(p_{\xi}\bigr)
		+\tfrac{1-\alpha}{2}\partial_{i}\ln\bigl(p_{\xi}\bigr)\cdotp\partial_{j}\ln\bigl(p_{\xi}\bigr)\bigr)
		\partial_{k}\ln\bigl(p_{\xi}\bigr)\bigr],\label{eq:48}\nonumber
	\end{alignat*}
	where $\mathbb{E}_{p_{\xi}}$ denotes the mean, or expectation, with respect to the probability 
	$p_{\xi}dx$, and where $\partial_{i}$ is a shorthand for $\tfrac{\partial}{\partial\xi_{i}}$. 
	In the formulas above, it is assumed that the function $p_{\xi}(x)$ is smooth with respect to 
	$\xi$ and that the expectations are finite. When the first formula above defines a smooth metric $h_{F}$ on $S$, 
	it is then called \textit{Fisher metric}. In this case, the 
	$\Gamma_{ij,k}^{(\alpha)}$'s define a connection $\nabla^{(\alpha)}$ on $S$ via the formula 
	$\Gamma_{ij,k}^{(\alpha)}(\xi)=(h_F)_{\xi}(\nabla_{\partial_{i}}^{(\alpha)}\partial_{j},\partial_{k})$, 
	which is called the \textit{$\alpha$-connection}.

%
	Among the $\alpha$-connections, the $1$-connection is particularly important and 
	is usually referred to as the \textit{exponential connection}, also denoted by $\nabla^{(e)}$. 
	In this paper, we will only consider statistical manifolds $S$ for which the Fisher metric $h_{F}$ and
	exponential connection $\nabla^{(e)}$ are well defined.


	We now recall the definition of an exponential family. 
\begin{definition}\label{def:5.3} 
	An \textit{exponential family} $\mathcal{E}$ on a measure space $(\Omega,dx)$ is a set of probability 
	density functions $p(x;\theta)$ of the form 
	$p(x;\theta)=\exp\big\{ C(x)+\sum_{i=1}^{n}\theta_{i}F_{i}(x)-\psi(\theta)\big\},$ 
	where $C,F_1...,F_n$ are measurable functions on $\Omega$, $\theta=(\theta_{1},...,\theta_{n})$ 
	is a vector varying in an open subset $\Theta$ of $\mathbb{R}^{n}$ and where $\psi$ is a function defined on $\Theta$.
\end{definition}
	In the above definition, it is assumed that the family of functions $\{1,F_1,...,$ $F_n\}$ 
	is linearly independent, so that the map $p(x,\theta)\mapsto\theta$ becomes a bijection, hence defining a global 
	chart for $\mathcal{E}$. The parameters $\theta_{1},...,\theta_{n}$ are called the 
	\textit{natural} or \textit{canonical parameters} of the exponential family $\mathcal{E}$. 
	The function $\psi$ is known as the \textit{log partition function} or \textit{cumulant function}.

\begin{example}[\textbf{Binomial distribution}]\label{exa:5.6}
	Let $\mathcal{B}(n)$ be the set of Binomial distributions $p(k)=\binom{n}{k}q^{k}\bigl(1-q\bigr)^{n-k}$, $q\in (0,1)$, 
	defined over $\Omega=\{0,...,n\}$. It is a 1-dimensional exponential family, because 
	$p(k)=\exp\big\{ C(k)+\theta F(k)-\psi(\theta)\big\}$, where $\theta=\ln\big(\frac{q}{1-q}\big)$, $C(k)= \ln\binom{n}{k}$, 
	$F(k)=k$, $k\in \Omega$, and $\psi(\theta)=n\ln\big(1+\exp(\theta)\big)$. 
\end{example}
\begin{example}[\textbf{Categorical distribution}]\label{exa:5.5}
	Let $\Omega=\{ x_{1},...,x_{n}\}$ be a finite set and let $\mathcal{P}_{n}^{\times}$ be the set of maps
	$p:\Omega\to \mathbb{R}$ satisfying $p(x)>0$ for all $x\in \Omega$ and $\sum_{x\in \Omega}p(x)=1$. Then 
	$\mathcal{P}_{n}^{\times}$ is an exponential family of dimension $n-1$. Indeed, elements in $\mathcal{P}_{n}^{\times}$ 
	can be parametrized as follows: $p(x;\theta)=\exp\big\{\sum_{i=1}^{n-1}\theta_{i}F_{i}(x)-\psi(\theta)\big\}$,
	where $x\in \Omega$, $\theta=(\theta_{1},...,\theta_{n-1})\in\mathbb{R}^{n-1},$ 
	$F_{i}(x_{j})=\delta_{ij}$ \textup{(Kronecker delta)} and $\psi(\theta)=\ln\big(1+\sum_{i=1}^{n-1}e^{\theta_{i}}\big)$. 
%
%
%
%
\end{example}

\begin{example}[\textbf{Poisson distribution}]
	A Poisson distribution is a distribution over $\Omega=\mathbb{N}=\{0,1,...\}$ of the form 
	$p(k;\lambda)=e^{-\lambda}\tfrac{\lambda^{k}}{k!}$, 
	where $k\in \mathbb{N}$ and $\lambda>0$. Let $\mathscr{P}$ denote the set of all Poisson distributions 
	$p(\,.\,;\lambda)$, $\lambda>0$. The set $\mathscr{P}$ is an exponential family, because 
	$p(k,\lambda)=\textup{exp}\big(C(k)+F(k)\theta-\psi(\theta)\big),$ where 
	$C(k)=-\ln(k!)$, $F(k)=k$, $\theta=\ln(\lambda)$, and $\psi(\theta)=\lambda=e^{\theta}.$
\end{example}

\begin{example}[\textbf{Negative Binomial distribution}]
	Let $\Omega=\mathbb{N}=\{0,1,....\}$ and $r\in \mathbb{N}$, $r\geq 1$. 
	Let $\mathcal{NB}(r)$ denote the set of functions $p:\Omega\to \mathbb{R}$ of the form
	$p(k;q)=\binom{k+r-1}{r-1}(1-q)^{k}q^{r}$, where $k=0,1,2,...$, $q\in (0,1)$ and 
	$\binom{k+r-1}{r-1}=\tfrac{(k+r-1)!}{k!(r-1)!}$. Using the fact that 
	$\tfrac{1}{(1-t)^{r}}=\sum_{k\geq 0}\binom{k+r-1}{r-1}t^{k}$ for $|t|<1$, 
	one sees that $\sum_{k\geq 0}p(k;q)=1$. Each element of $\mathcal{NB}(r)$ is a called a \textit{negative Binomial distribution}.  
	The set $\mathcal{NB}(r)$ is an exponential family, because $p(k;q)=p(k;\theta)=\exp\big\{ C(k)+\theta F(k)-\psi(\theta)\big\}$,
	where $\theta=\ln(1-q)\in (-\infty,0),$ $C\bigl(k\bigr)= \ln\binom{k+r-1}{r-1}$, $F(k)= k$ and 
	$\psi(\theta)= -r\ln\bigl(1-e^{\theta}\bigr)$. 
\end{example}

	Let $\mathcal{E}$ be an exponential family, with Fisher metric $h_{F}$, exponential connection $\nabla^{(e)}$, 
	natural parameters $\theta=(\theta_{1},...,\theta_{n})$ and log partition function $\psi=\psi(\theta)$. 
	Under mild assumptions, it can be proved that:
	\begin{itemize}
	\item $(\mathcal{E},h_{F},\nabla^{(e)})$ is a dually flat manifold. 
	\item The matrix representation of $h_{F}$ with respect to the basis 
		$\{\tfrac{\partial}{\partial \theta_{1}},...,\tfrac{\partial}{\partial \theta_{n}}\}$ 
		is the Hessian of $\psi$, that is, 
		\begin{eqnarray*}
			(h_{F})_{ij}=h_{F}\bigg(\dfrac{\partial}{\partial \theta_{i}}, 
			\dfrac{\partial}{\partial \theta_{j}}\bigg)=\dfrac{\partial^{2}\psi}{\partial \theta_{i}\partial \theta_{j}}
		\end{eqnarray*}
	for all $1\leq i,j\leq n$. 
	\end{itemize}
	This is the case, for example, 
	if $\Omega$ is a finite set endowed with the counting measure (see \cite{shima}, Chapter 6). 
	In the sequel, we will always regard an exponential family  $\mathcal{E}$ as a dually flat manifold. 

	Since an exponential family is dually flat, it is natural to ask whether it is toric. 
	Below we describe four examples. Let $\mathbb{P}_{n}(c)$ be the complex projective space of complex dimension $n$, 
	endowed with the Fubini-Study metric normalized 
	in such a way that the holomorphic sectional curvature is $c>0$. Let $\Phi_{n}$ be the action of 
	$\mathbb{T}^{n}$ on $\mathbb{P}_{n}(c)$ defined by 
	\begin{eqnarray}\label{neknknknknfdkn}
		\Phi_{n}([t],[z])= [e^{2i\pi t_{1}}z_{1},...,e^{2i\pi t_{n}}z_{n},z_{n+1}]. 
		\,\,\,\,\,\,\,(\textup{homogeneous coordinates})
	\end{eqnarray}

	Given $c<0$, we will denote by $\mathbb{D}(c)$ the unit disk in $\mathbb{C}$ endowed with the Hyperbolic metric 
	$ds^2= -4c^{-1}\left(dx^{2}+dy^{2} \right)(1-x^{2}-y^{2})^{-2}$.

\begin{proposition}\label{nfeknwknwknk}
	In each case below, the torus action is a regular torification of the indicated exponential family $\mathcal{E}$. \\

	\begin{tabular}{lll}
		$(a)$ &  $\mathcal{E}=\mathcal{P}_{n+1}^{\times}$,    &    
						$\Phi_{n}:\mathbb{T}^{n}\times \mathbb{P}_{n}(1)\to \mathbb{P}_{n}(1)$.\\[0.4em]
		$(b)$ & $\mathcal{E}=\mathcal{B}(n)$,                &    $\Phi_{1}:\mathbb{T}^{1}\times 
								\mathbb{P}_{1}(\tfrac{1}{n})\to \mathbb{P}_{1}(\tfrac{1}{n})$.\\[0.4em]
		$(c)$ & $\mathcal{E}=\mathscr{P}$,                   & $\mathbb{T}^{1}\times \mathbb{C}\to \mathbb{C}$, $([t],z)\mapsto e^{2i\pi t}z$.\\[0.4em]	
		$(d)$ & $\mathcal{E}=\mathcal{NB}(r)$,               & $\mathbb{T}^{1}\times \mathbb{D}(-\tfrac{1}{r})\to \mathbb{D}(-\tfrac{1}{r})$, 
		$([t],z)\mapsto e^{2i\pi t}z$.	

	\end{tabular}
\end{proposition}

	\noindent For a proof and more examples, see \cite{molitorToric} (see also Section \ref{nceknwkenfknk} below).

\subsection{Hessian sectional curvature.}\label{sec:1}

	The notion of dually flat manifold was originally introduced by H. Shima in a 
	different (but equivalent) form, under the name \textit{Hessian manifold} \cite{shima3,shima}.
	In this section, we briefly recall Shima's definition, and discuss the notion of 
	\textit{Hessian sectional curvature}, which is a real-analogue of the notion 
	of holomorphic sectional curvature in K\"{a}hler geometry. Our main reference is \cite{shima}.

	We begin by a general discussion concerning affine geometry. An \textit{affine manifold} 
	is a pair $(M,\nabla)$, where $M$ is a manifold and $\nabla$ is a torsion-free flat connection. 
	An \textit{affine coordinate system} on an affine 
	manifold $(M,\nabla)$ is a coordinate system $(x_{1},...,x_{n})$ defined on some open set $U\subseteq M$ such that 
	$\nabla_{\partial_{i}}\partial_{j}=0$ for all $i,j=1,...,n$, where $\partial_{i}=\tfrac{\partial}{\partial x_{i}}$. 
	It can be shown that for every point $p$ in an affine manifold $M$, there is an affine coordinate system 
	$(x_{1},...,x_{n})$ defined on some neighborhood $U\subseteq M$ of $p$ (see \cite{shima}). 

\begin{definition}\label{nfekwnkefnknkhh}
	Let $(M,\nabla)$ be an affine manifold. A Riemannian metric $g$ on $M$ is said to be a \textit{Hessian metric} 
	if for every $p\in M$, there is an affine coordinate system $(x_{1},...,x_{n})$ defined on some open 
	neighborhood $U\subseteq M$ of $p$ and a smooth function $\varphi:U\to \mathbb{R}$ such that 
        \begin{eqnarray*}
		g_{ij}=\dfrac{\partial^{2}\varphi}{\partial x_{i}\partial x_{j}}
	\end{eqnarray*}
	on $U$ for all $i,j\in \{1,...,n\}$, where $g_{ij}=g(\tfrac{\partial}{\partial x_{i}},\tfrac{\partial}{\partial x_{j}})$. In this case, 
	the pair $(\nabla,g)$ is called a \textit{Hessian structure} on $M$, and $\varphi$ is said to be a \textit{potential} of $(\nabla,g)$. 
	A manifold $M$ endowed with a Hessian structure $(\nabla,g)$ is called a \textit{Hessian manifold}, and is denoted by $(M,\nabla,g)$. 
\end{definition}

	The following result shows that the notions of Hessian manifolds and dually flat manifolds are equivalent
	(see, e.g., \cite{shima}, Corollary 2.1 and \cite{Amari-Nagaoka}, Section 3.3).
\begin{lemma}
	Let $(M,\nabla)$ be an affine manifold equipped with a Riemannian metric $g$. Then $(M,\nabla,g)$ is a Hessian manifold 
	if and only if $(M,g,\nabla)$ is a dually flat manifold. 
\end{lemma}
	
	Let $(M,\nabla,g)$ be a Hessian manifold. We will denote by $\gamma$ 
	the tensor field of type $(1,2)$ on $M$ defined by 
	\begin{eqnarray*}
		\gamma(X,Y)=\nabla^{R}_{X}Y-\nabla_{X}Y,
	\end{eqnarray*}
	where $\nabla^{R}$ is the Riemannian connection associated to $g$. 

	The \textit{Hessian curvature tensor}, denoted by $Q$, 
	is the covariant differential of $\gamma$ with respect to $\nabla$. Thus $Q=\nabla \gamma$ is a tensor field of 
	type $(1,3)$ and we have
	\begin{eqnarray*}
		Q(X,Y,Z)=(\nabla\gamma)(X,Y,Z)=\nabla_{Z}\big(\gamma(X,Y)\big)-\gamma(\nabla_{Z}X,Y)-\gamma(X,\nabla_{Z}Y)
	\end{eqnarray*}
	for all vector fields $X,Y,Z$ on $M$. 

	In what follows, we will use standard tensorial notation and Einstein's summation convention.
	So, for instance $Q(\tfrac{\partial}{\partial x_{j}},\tfrac{\partial}{\partial x_{k}},\tfrac{\partial}{\partial x_{l}})
	=\sum_{i=1}^{n}Q^{i}_{\textup{}\,\,jkl}\tfrac{\partial}{\partial x_{i}}=Q^{i}_{\textup{}\,\,jkl}\tfrac{\partial}{\partial x_{i}}$.

	Let $\widehat{Q}$ be the endomorphism on the space of symmetric contravariant 
	tensor fields of degree 2 defined by 
	\begin{eqnarray}\label{nekdwnkvnknkn}
		\widehat{Q}(\xi)=g^{ka}Q^{i}_{\textup{}\,\,jal}\xi^{jl} 
		\dfrac{\partial}{\partial x_{i}}\otimes \dfrac{\partial}{\partial x_{k}}, 
	\end{eqnarray}
	where $(x_{1},...,x_{n})$ is an affine coordinate system on $U\subseteq M$ with potential $\varphi:U\to \mathbb{R}$,
	$\xi=\xi^{ik}\tfrac{\partial}{\partial x_{i}}\otimes \tfrac{\partial}{\partial x_{k}}$ and $g^{ka}$ is the $(k,a)$-entry 
	of the inverse of the matrix representation of the metric $g$ in the coordinates $(x_{1},...,x_{n})$. 

	The \textit{Hessian sectional curvature} of a nonzero symmetric contravariant 
	tensor $\xi_{p}$ of degree 2 at $p\in M$ is the real number
	\begin{eqnarray}\label{nvejbjefbjnkdnk}
		S(\xi_{p})=\dfrac{\langle \widehat{Q}(\xi_{p}),\xi_{p}\rangle}{\langle \xi_{p},\xi_{p}\rangle}, 
	\end{eqnarray}
	where $\langle \xi,\eta\rangle=\xi^{ij}g_{ia}g_{jb}\eta^{ab}$ is the inner product on the space of 
	contravariant tensor fields of degree 2 induced by $g$. 

	If $S(\xi_{p})$ is a constant $c$ for all symmetric contravariant
	tensor field $\xi_{p}\neq 0$ of degree 2 and for all $p\in M$, then $(M, \nabla, g)$ 
	is said to be a \textit{Hessian manifold of constant Hessian sectional curvature $c$}.

\begin{proposition}[\cite{shima}, Corollary 3.1]\label{nefknkekfwnk}
	Let $(M,\nabla,g)$ be a Hessian manifold. Suppose that $TM$ is endowed with the 
	K\"{a}hler structure associated to Dombrowski's construction. The following conditions are equivalent.
	\begin{enumerate}[(a)]
		\item The Hessian sectional curvature of $(M,\nabla,g)$ is a constant $c$. 
		\item The holomorphic sectional curvature of $TM$ is a constant $-c$. 
	\end{enumerate}
\end{proposition}

	It should be noted that the Hessian sectional curvature is also well-defined for Hessian manifolds of dimension 1, 
	in which case we have the following

\begin{proposition}[\cite{Molitor-Hessian}, Proposition 1]\label{nfendwnfeknkdwnkn}
	Let $(M,\nabla,g)$ be a Hessian manifold of dimension 1 with Hessian sectional curvature $S$. 
	Let $x:U\to \mathbb{R}$ be an affine coordinate system defined 
	on an open connected set $U\subseteq M$. Then for any nonzero symmetric contravariant 
	tensor $\xi_{p}$ of degree 2 at $p\in U$, 
	\begin{eqnarray}\label{enknknkrfnkefnknk}
		S(\xi_{p})=\dfrac{1}{2g}\dfrac{\partial^{2}\ln(g)}{\partial x^{2}},
	\end{eqnarray}
	where $g(x)=g(\tfrac{\partial}{\partial x}, \tfrac{\partial}{\partial x})$. 
\end{proposition}
	
	It follows from Proposition \ref{nfendwnfeknkdwnkn} that the Hessian sectional curvature in dimension 1 
	is just a smooth function $S:M\to \mathbb{R}$ whose local expression is given by the right hand side of \eqref{enknknkrfnkefnknk}.
	Moreover, we have the following 

\begin{proposition}[\cite{Molitor-Hessian}, Proposition 3]\label{nkwknkrnkwdnknkn}
	Let $(M,\nabla,g)$ be a Hessian manifold of dimension 1 with Hessian sectional curvature $S:M\to \mathbb{R}$. 
	Let $x:U\to \mathbb{R}$ be an affine coordinate system with respect to $\nabla$ defined on a 
	connected open subset $U$ of $M$. If the Hessian sectional curvature $S$ is constant and equal to $\lambda\in \mathbb{R}$, 
	then there are real numbers $a$ and $b$ such that 
	the local expression of $g(x)=g(\tfrac{\partial}{\partial x},\tfrac{\partial}{\partial x})$ on $U$ is given by
	\begin{enumerate}[(a)]
	\item $g(x)=e^{ax+b}$\,\,\,if $\lambda=0.$
	\item $g(x)=\tfrac{a}{\lambda \cos^{2}\big(\sqrt{a}x+b\big)}$ or 
		$\tfrac{a}{\lambda \sinh^{2}\big(\sqrt{a}x+b\big)}$ 
		or $\tfrac{1}{\lambda(x+b)^{2}}$ \,\,\, if $\lambda>0.$
	\item $g(x)=\tfrac{-a}{\lambda \cosh^{2}\big(\sqrt{a}x+b\big)}$\,\,\, if $\lambda<0.$
	\end{enumerate}
\end{proposition}

	Note that the constant $a$ must be positive when $\lambda\neq 0$ (since $g>0$).

\subsection{Exponential families with constant Hessian sectional curvature.} \label{infnekfkenkfnk}

	In this section, we discuss a special case of Proposition \ref{nkwknkrnkwdnknkn} in which $M=\mathcal{E}$ is
	a 1-dimensional exponential families defined over a finite sample space. The material is taken from \cite{Molitor-Hessian}. 

	We start by establishing some notation. Let $\Omega=\{x_{0},...,x_{m}\}$ be a finite set, and let 
	$C(\Omega)$ be the space of maps $\Omega\to \mathbb{R}$ (there is a natural identification 
	$C(\Omega)\cong \mathbb{R}^{m+1}$). Given $F,C\in C(\Omega)$, we will denote by 
	$\mathcal{E}_{C,F}$ the 1-dimensional exponential family defined over $\Omega$, with elements of the form 
	\begin{eqnarray*}
		 p_{C,F}(x;\theta)=\textup{exp}\big\{C(x)+\theta F(x)-\psi_{C,F}(\theta)\big\},
	\end{eqnarray*}
	where $x\in \Omega$, $\theta\in \mathbb{R}$ and $\psi_{C,F}(\theta)=\ln\big(\sum_{k=0}^{m}\textup{exp}(C(x_{k})
	+\theta F(x_{k}))\big).$ In the notation above, it is assumed that the function $F$ and the constant function 
	$\mathds{1}:\Omega\to \mathbb{R},$ $x\mapsto 1$ are linearly independent (this guarantees that the map 
	$\mathbb{R}\to \mathcal{E}_{C,F}$, $\theta\mapsto p_{C,F}(\,.\,;\theta)$ is bijective). In other words, 
	$F\in \mathbb{R}^{m+1}-\mathbb{R}\cdot \mathds{1}$. 
\begin{definition}\label{kwwdefksn}
	Two 1-dimensional exponential families $\mathcal{E}_{C,F}$ and $\mathcal{E}_{C',F'}$ defined over the same 
	set $\Omega=\{x_{0},...,x_{m}\}$ are \textit{equivalent} if the families of maps 
	$\{p_{C,F}(\,.\,;\theta)\,:\,\Omega\to \mathbb{R}\}_{\theta\in \mathbb{R}}$ and 
	$\{p_{C',F'}(\,.\,;\theta)\,:\,\Omega\to \mathbb{R}\}_{\theta\in \mathbb{R}}$ coincide. In this case, we write 
	$\mathcal{E}_{C,F}\sim \mathcal{E}_{C',F'}$. 
\end{definition}
	
	In order to characterize equivalent exponential families, we introduce the group of matrices 
	\begin{eqnarray*}
		G:=\bigg\{
		 \begin{bmatrix}
			1 & b & d\\
			0 & a & c\\
			0 & 0 & 1
		\end{bmatrix}
		\,\,\,\,\bigg\vert\,\,\,\, a,b,c,d\in \mathbb{R}, \,\,a\neq 0 \bigg\}.
	\end{eqnarray*}
	Given an integer $m\geq 1$, the group $G$ acts on 
	$U_{m}:=\mathbb{R}^{m+1}\times (\mathbb{R}^{m+1}-\mathbb{R}\cdot\mathds{1})$ via the formula
	\begin{eqnarray*}
		  \begin{bmatrix}
			1 & b & d\\
			0 & a & c\\
			0 & 0 & 1
		\end{bmatrix}
		\cdot (C,F):= (C+bF+d\mathds{1}, aF+c\mathds{1}), 
	\end{eqnarray*}
	where $C\in \mathbb{R}^{m+1}$ and $F\in \mathbb{R}^{m+1}-\mathbb{R}\cdot\mathds{1}$. 
\begin{lemma}[\cite{Molitor-Hessian}, Lemma 5]\label{nknknkn}
	Two 1-dimensional exponential families $\mathcal{E}_{C,F}$ and $\mathcal{E}_{C',F'}$ defined over the same 
	finite set $\Omega$ are equivalent if and only if there exists $g=\Big[\begin{smallmatrix}
		1&b&d\\
		0&a&c\\
		0&0&1
	\end{smallmatrix} \Big]\in G$ such that $(C,F)=g\cdot (C',F')$. 
	In this case, the following hold. 
	\begin{enumerate}[(i)]
	\item $p_{C,F}(x;\theta)=p_{C',F'}(x;a\theta+b)$ \,\,\,\,\,for all $x\in \Omega$ and all $\theta\in \mathbb{R}$. 
	\item $\psi_{C,F}(\theta)=\psi_{C',F'}(a\theta+b)+c\theta+d$ \,\,\,\,\,for all $\theta\in \mathbb{R}$. 
	\end{enumerate}
\end{lemma}

\begin{remark}\label{nefkwnkneknk}
	A consequence of the Lemma \ref{nknknkn} is that equivalent exponential families are isomorphic. More precisely, if 
	$\mathcal{E}_{C,F}\sim \mathcal{E}_{C',F'}$, and if $(C,F)=g\cdot (C',F')$, where
	$g=\Big[\begin{smallmatrix}
		1&b&d\\
		0&a&c\\
		0&0&1
	\end{smallmatrix} \Big]\in G$, then the map 
	\begin{eqnarray*}
		\mathcal{E}_{C,F}\to \mathcal{E}_{C',F'}, \,\,\,\,\,\,\,\,p_{C,F}(\,.\,;\theta)\mapsto p_{C',F'}(\,.\,;a\theta+b), 
	\end{eqnarray*}
	is an isomorphism of dually flat manifolds. (To see that it is an isometry, use (ii) and 
	the fact that the matrix representation of the Fisher metric is the Hessian of the log partition function.)
\end{remark}

\begin{definition}\label{nckwdknknk}
	Let $\mathcal{E}=\mathcal{E}_{C,F}$ be a 1-dimensional exponential family defined over a finite set 
	$\Omega=\{x_{0},...,x_{m}\}$. Let $\{\alpha_{i}\}_{i=0,...,p}$ and $\{\omega_{i}\}_{i=0,...,p}$ be the families 
	of real numbers characterized by the following conditions: 
	\begin{enumerate}[(i)]
		\item $F(\Omega)=\{\alpha_{0},...,\alpha_{p}\}$ and $\alpha_{0}<...<\alpha_{p}$, 
		\item $e^{\omega_{i}}=\sum_{x\in F^{-1}(\alpha_{i})}e^{C(x)}$. 
	\end{enumerate}
	Let $\Omega_{red}=\{0,1,...,p\}$. Define $C_{red},F_{red}:\Omega_{red}\to \mathbb{R}$ by 
	\begin{eqnarray*}
		F_{red}(k)=\alpha_{k} \,\,\,\,\,\,\,\textup{and} \,\,\,\,\,\,\, C_{red}(k)=\omega_{k},
	\end{eqnarray*}
	where $k\in \{0,...,p\}$. Then $\mathcal{E}_{red}:=\mathcal{E}_{C_{red},F_{red}}$ is a 1-dimensional 
	exponential family defined over $\Omega_{red}$. We call it the \textit{reduced exponential family} of 
	$\mathcal{E}$. 
\end{definition}
\begin{remark}\label{nrvkknkwenken}
	If $\mathcal{E}=\mathcal{E}_{C,F}$, then 
	the map $\mathcal{E}\to \mathcal{E}_{red}$, $p_{C,F}(\,.\,;\theta)\mapsto p_{C_{red},F_{red}}(\,.\,;\theta)$ 
	is an isomorphism of dually flat manifolds. This follows from the formula 
	$\psi_{C,F}(\theta)=\psi_{C_{red},F_{red}}(\theta)$ ($\theta\in \mathbb{R}$) and the 
	fact that the matrix representation of the Fisher metric is the Hessian of 
	the log partition function.
\end{remark}

	Recall that $\mathcal{B}(n)$ denotes the set of all binomial distributions 
	$p(k)=\binom{n}{k}q^{k}(1-q)^{n-k}$, $q\in (0,1)$,
	defined over $\Omega=\{0,1,...,n\}$.
	
\begin{proposition}[\cite{Molitor-Hessian}, Proposition 7]\label{nkdnkenkknk}
	Let $\mathcal{E}=\mathcal{E}_{C,F}$ be a 1-dimensional exponential family defined over a finite set 
	$\Omega=\{x_{0},...,x_{m}\}$. The following are equivalent.
	\begin{enumerate}[(i)]
	\item The Hessian sectional curvature of $\mathcal{E}$ is constant. 
	\item $\mathcal{E}_{red}\sim \mathcal{B}(p)$, where $p+1$ is the cardinality of $\Omega_{red}$. 
	\end{enumerate}
	If any one of the conditions holds, then the Hessian sectional curvature of $\mathcal{E}$ is $-\tfrac{1}{p}$. 
\end{proposition}

\section{Dually flat models}\label{nceknwkenfknk}

	The purpose of this section is to show that every complete and simply connected 
	1-dimensional complex space form can be realized as the regular 
	torification of an appropriate dually flat manifold (see Proposition \ref{nekwnkendkwnkn}). 
	Let $c$ be a real number and let $M_{c}$ be the manifold defined by 
	\begin{eqnarray*}
		M_{c}=
		\left\lbrace
		\begin{array}{lll}
			\mathbb{R} &\textup{if}& c\geq 0, \\
			(-\infty,0) &\textup{if}& c<0. 
		\end{array}
		\right.
	\end{eqnarray*}
	Let $\theta$ denote the usual linear coordinate on $\mathbb{R}$ (or an open subset of it), and 
	let $h_{c}$ be the Riemannian metric on $M_{c}$ whose local expression in the coordinate $\theta$ is given by 
	\begin{eqnarray}\label{nveknkefnkn}
		h_{c}(\theta)=
		\left\lbrace
		\begin{array}{lll}
			\tfrac{1}{c\cosh^{2}(\theta)}   &\textup{if}& c > 0,  \\[0.5em]
			 e^{\theta}                     &\textup{if}& c=0,   \\[0.5em]
			\tfrac{-1}{c\sinh^{2}(\theta)}  &\textup{if}& c<0.
		\end{array}
		\right.
	\end{eqnarray}

	Let $\nabla$ be the flat connection on $M_{c}\subseteq \mathbb{R}$. 

\begin{lemma}\label{ncknknkefnknknk}
	$(M_{c},h_{c},\nabla)$ is a dually flat manifold whose Hessian sectional curvature is constant and equal to $-c$. 
\end{lemma}
\begin{proof}
	By a direct verification using Proposition \ref{nfendwnfeknkdwnkn}. 
\end{proof}

\begin{example}
	The exponential families $\mathcal{B}(n)$, $\mathscr{P}$ and $\mathcal{NB}(r)$ (regarded as dually flat manifolds)
	are isomorphic to $M_{1/n}$, $M_{0}$ and $M_{-1/r}$, respectively. To see this, compute the Hessian of each 
	log partition function and compare with \eqref{nveknkefnkn}. 
\end{example}

	Since $(M_{c},h_{c},\nabla)$ is dually flat, the tangent bundle $TM_{c}$ is naturally a K\"{a}hler manifold by 
	Dombrowski's construction. Let $g_{c}$ and $\omega_{c}$ be the corresponding K\"{a}hler metric and form on $TM_{c}$. 

	We will identify the tangent bundle of $M_{c}$ with $M_{c}\times \mathbb{R}\subseteq \mathbb{C}$ via the map 
	$y\tfrac{\partial}{\partial \theta}\big\vert_{x}\mapsto z=x+iy$. Then 
	\begin{eqnarray*}
		&& (g_{c})_{z}(u,v)=h_{c}(x)(u_{1}v_{1}+u_{2}v_{2})\\
		&\textup{and}&(\omega_{c})_{z}(u,v)= h_{c}(x)(u_{1}v_{2}-u_{2}v_{1}), 
	\end{eqnarray*}
	where $z=x+iy,u=u_{1}+iu_{2},v=v_{1}+iv_{2}$, $x,u_{1},v_{1}\in M_{c}$ and $y,u_{2},v_{2}\in \mathbb{R}$. 

	Let $F(c)$ be a complete, connected and simply connected 1-dimensional complex space form of 
	constant holomorphic sectional curvature $c$. Thus 
	\begin{eqnarray*}
		F(c)=
		\left\lbrace
		\begin{array}{lll}
			\mathbb{D}(c)                    &\textup{if}& c<0,   \\[0.5em]
			\mathbb{C}                    &\textup{if}& c=0,   \\[0.5em]
			\mathbb{P}_{1}(c) &\textup{if}& c > 0,
		\end{array}
		\right.
	\end{eqnarray*}
	where $\mathbb{D}(c)=\{z\in \mathbb{C}\,\,\vert\,\,|z|<1\}$ is the unit disk endowed with the Hyperbolic metric 
	$ds^2= -4c^{-1}\left(dx^{2}+dy^{2} \right)(1-x^{2}-y^{2})^{-2}$, and $\mathbb{P}_{1}(c)$ is 
	the complex projective space of complex dimension 1 and holomorphic sectional curvature $c>0$.

	Let $\mathbb{T}=\mathbb{R}/\mathbb{Z}$ be the real torus of dimension $1$. Given $t\in \mathbb{R}$, we will 
	denote by $[t]$ the equivalence class of $t$ in $\mathbb{T}$. Let 
	$\Phi_{c}:\mathbb{T}\times F(c)\to F(c)$ be the torus action defined via the formulas:
	\begin{eqnarray*}
		\left\lbrace
		\begin{array}{lll}
			\Phi_{c}([t],z)=e^{2i\pi t}z                           & \textup{if} & c\leq 0, \\[0.5em]
			\Phi_{c}([t],[z_{0},z_{1}])=[e^{2i\pi t}z_{0},z_{1}]    & \textup{if} & c> 0.
		\end{array}
		\right.
	\end{eqnarray*}

	We will use $F(c)^{\circ}$ to denote the set of points $p\in F(c)$ where the action $\Phi_{c}$ is free.
	Note that $F(c)^{\circ}=F(c)-\{0\}$ if $c \leq 0$ and $F(c)^{\circ}=
	\{[z_{0},z_{1}]\in\mathbb{P}_{1}(c)\,\,\vert\,\,z_{0}z_{1}\neq 0\}$ if $c>0$.

\begin{proposition}\label{nekwnkendkwnkn}
	Given $c\in \mathbb{R}$, the space form $F(c)$, equipped with the torus action $\Phi_{c}$, is a regular torification 
	of $(M_{c},h_{c},\nabla)$.
\end{proposition}

	The rest of this section is devoted to the proof of Proposition \ref{nekwnkendkwnkn}. First we need some notation. 
	Let $\tau_{c}:TM_{c}\to F(c)^{\circ}$ be the map defined by 
	\begin{eqnarray*}
		\tau_{c}(z)=
		\left\lbrace
		\begin{array}{lll}
			e^{z}     & \textup{if} & c< 0,\\[0.3em]
			2e^{z/2}   & \textup{if} & c= 0,\\[0.3em]
			\textit{}[e^{z},1] & \textup{if} & c>0.
		\end{array}
		\right.
	\end{eqnarray*}

	Let $\epsilon:\mathbb{R}\to \{1,2\}$ be the function defined by $\epsilon(c)=2$ if $c=0$ and 
	$\epsilon(c)=1$ otherwise.

\begin{lemma}\label{ndkenknfedkwnk}
	Let $c\in \mathbb{R}$ be arbitrary. The map $\tau_{c}$ is a K\"{a}hler covering map whose Deck transformation group is 
	$\textup{Deck}(\tau_{c})=\{\varphi_{k}\,\,\vert\,\,
	k\in \mathbb{Z}\}$, where $\varphi_{k}:TM_{c}\to TM_{c}$ is defined by $\varphi_{k}(z)= z+2ki\pi \epsilon(c)$.
\end{lemma}

	We begin with the case $c>0$. 
	Let $g_{FS}$ and $\omega_{FS}$ denote the Fubini-Study metric and symplectic 
	form on $\mathbb{P}_{1}(4)$, respectively (thus the corresponding holomorphic sectional curvature on 
	the complex projective space is $4$). Note that $\tfrac{4}{c}g_{FS}$ is the Fubini-Study metric on $\mathbb{P}_{1}(c)$ 
	whose corresponding holomorphic sectional curvature is $c$. 

	Let $S^{3}=\{(z_{1},z_{2})\in \mathbb{C}^{2}\,\,\vert\,\,\,|z_{1}|^{2}+|z_{2}|^{2}=1\}$ be the unit 
	sphere of dimension $3$ in $\mathbb{R}^{4}=\mathbb{C}^{2}$, and let 
	$\pi_{H}:S^{3}\to \mathbb{P}_{1}(4)$ be the Hopf fibration. Thus 
	\begin{eqnarray*}
		\pi_{H}(z_{1},z_{2})=[z_{1},z_{2}] \,\,\,\,\,\,\,\,\,\,(\textup{homogeneous coordinates}). 
	\end{eqnarray*}
	It is well-known that $\pi_{H}^{*}\omega_{FS}=j^{*}\omega_{\textup{can}}$, where $j:S^{3}\hookrightarrow \mathbb{C}^{2}$ 
	is the inclusion map and $\omega_{\textup{can}}$ is the canonical symplectic form on $\mathbb{R}^{4}=\mathbb{C}^{2}$. 
	Recall that if $\langle z,w\rangle=\overline{z}_{1}w_{1}+\overline{z}_{2}w_{2}$ denotes the usual Hermitian 
	product on $\mathbb{C}^{2}$, then 
	\begin{eqnarray*}
	      (\omega_{\textup{can}})_{z}(u,v) & = & \textup{Im}\langle u,v\rangle,	 
	\end{eqnarray*}
	where $z,u,v\in \mathbb{C}^{2}$, and where $\textup{Im}(w)$ is the imaginary part of $w\in \mathbb{C}$. 
\begin{proof}[Proof of Lemma \ref{ndkenknfedkwnk} ($c>0$)]
	Let $U=\{ [z_0, z_1] \in \mathbb{P}_1(c)\,\, \vert \,\, z_1\neq 0\}$, and let $\psi: U   \longrightarrow \mathbb{C}$ 
	be defined by $\psi([z_0,z_1])=z_0z_1^{-1}.$ Then the pair $(U,\psi)$ is a complex chart for $\mathbb{P}_1(c)$.
	If we identify the complex manifolds $\mathbb{P}_{1}(c)^{\circ}$ and $\mathbb{C}^{*}=\mathbb{C}-\{0\}$ 
	via the map $\psi,$ then $\tau_{c}$ is just the map 
	$\mathbb{C}\to \mathbb{C}^{*}$, $z\mapsto e^{z}$, which shows that $\tau_{c}$ is holomorphic and a covering map, 
	whose Deck transformation group is $\{\varphi_{k}\,\,\vert\,\,
	k\in \mathbb{Z}\}$ (see Lemma \ref{ndkenknfedkwnk}). It remains to  show that $\tau_{c}$ is symplectic, that 
	is, $\tau_{c}^{*}(\tfrac{4}{c}\omega_{FS})=\omega_{c}$. Write $\tau_{c}=\pi_{H}\circ f$, where $\pi_{H}$ is the 
	Hopf map and where $f:\mathbb{C}=TM_{c}\longrightarrow \mathbb{S}^3$ is the map defined by $f(z)
	=\left(1 + e^{2\textup{Re}(z)}\right)^{-1/2}(e^z,1)$. Since $\pi_{H}^{*}\omega_{FS}=j^{*}\omega_{\textup{can}}$, 
	we see that 
	\begin{eqnarray*}
		\tau^{\ast}_{c}\omega_{FS}  =  (\pi_{H} \circ f)^{\ast}\omega_{FS} =  (j \circ f)^{\ast}\omega_{\textup{can}}  
		= F^{\ast}\omega_{\textup{can}},
	\end{eqnarray*}
	where $F:=j\circ f:\mathbb{C}\to \mathbb{C}^{2}$. Let $z,u,v\in \mathbb{C}$ be arbitrary. A simple calculation shows that 
         \begin{eqnarray*}
             F_{\ast_z}u  = \dfrac{(ue^z, 0)}{\left(1+ e^ {2\textup{Re}(z)}\right)^{\frac{1}{2}}} 
		 - \dfrac{e^{2\textup{Re}(z)}\textup{Re}(u)}{\left(1 +e^{2\textup{Re}(z)}\right)^{\frac{3}{2}}}(e^z, 1),
         \end{eqnarray*}
        from which it follows that 
	\begin{eqnarray}
		\lefteqn{(F^{*}\omega_{\textup{can}})_{z}(u,v)=
		\textup{Im}\langle F_{*_{z}}u,F_{*_{z}}v\rangle}\nonumber\\
		&=&\dfrac{1}{4\textup{cosh}^{2}(\textup{Re}(z))}\Big[\textup{Re}(u)\textup{Im}(v)-
		\textup{Im}(u)\textup{Re}(v)\Big],\label{nekwkenfknwk}
	\end{eqnarray}
	where we have used the formula $\tfrac{e^{\theta}}{(1+e^{\theta})^{2}}=\tfrac{1}{4\textup{cosh}^{2}(\theta/2)}$, $\theta\in \mathbb{R}$. 
	By inspection of Dombrowski's construction and the fact that 
	$h_{c}(\theta)=\tfrac{1}{c\textup{cosh}^{2}(\theta)}$, we see that the right-hand side of \eqref{nekwkenfknwk} 
	coincides with $(\omega_{4})_{z}(u,v)$. Thus 
	\begin{eqnarray*}
		\tau^{\ast}_{c}\omega_{FS}= F^{*}\omega_{\textup{can}}=\omega_{4}.
	\end{eqnarray*}
	Multiplying by $\lambda>0$, we get $\tau^{\ast}_{c}(\lambda\omega_{FS})=\lambda\omega_{4}=\omega_{4/\lambda}$. Taking 
	$\lambda=\tfrac{4}{c}$, we see that $\tau_{c}^{*}(\tfrac{4}{c}\omega_{FS})=\omega_{c}$, as desired. 
\end{proof}

\begin{proof}[Proof of Lemma \ref{ndkenknfedkwnk} ($c\leq 0$)] First suppose $c<0$. 
	Under the identification of $TM_{c}$ with $\{z\in \mathbb{C}\,\,\vert\,\,\textup{Re}(z)<0\}$, the map $\tau_{c}: TM_c 
	\rightarrow \mathbb{D}(c)^{\circ}=\mathbb{D}(c)-\{0\}$ is given by $\tau_{c}(z)=e^{z}$, which shows that $\tau_{c}$ is a holomorphic 
	covering map whose Deck transformation group is $\{\varphi_{k}\,\,\vert\,\,k\in \mathbb{Z}\}$ (see Lemma \ref{ndkenknfedkwnk}).
	To show that $\tau_{c}$ is symplectic, 
	it is enough to prove that $\tau_{c}$ is a local isometry. A straightforward calculation shows that the 
	pullback of the hyperbolic metric on $\mathbb{D}(c)^{\circ}$ by $\tau_{c}$ at $z=x+iy\in TM_{c}$ is 
	$(c\,\sinh^{2}(x))^{-1}(dx^{2}+dy^{2})$, which coincides with the K\"{a}hler metric $g_{c}$ on $TM_{c}$ associated to $h_{c}$ via 
	Dombrowski's construction. This completes the proof of the case $c<0$. The case $c=0$ is analogous. 
\end{proof}

	Given $c\in \mathbb{R}$, let $L_{c}\subset TM_{c}$ be the parallel lattice generated by 
	the vector field $2\pi\epsilon(c)\tfrac{\partial}{\partial \theta}$. Thus 
	\begin{eqnarray*}
		L_{c}=\big\{2k\pi\epsilon(c)\tfrac{\partial}{\partial \theta}\big \vert_{\theta} \,\,\big\vert\,\, 
		\theta \in M_{c},\,\, k \in \mathbb{Z} \big\}.
	\end{eqnarray*}

\begin{proof}[Proof of Proposition \ref{nekwnkendkwnkn}]
	Use Proposition \ref{nkwdnkkfenknk}, Lemma \ref{ndkenknfedkwnk} and the parallel lattice $L_{c}$. 
\end{proof}

\section{Global aspects}\label{nrfekwnkefnekn}

	The purpose of this section is to prove the following technical result:
\begin{proposition}\label{nfeknkefnknwk}
	Let $(M,h,\nabla)$ and $(M',h',\nabla')$ be toric dually flat manifolds. If $\textup{dim}(M)=\textup{dim}(M')$, then 
	every affine isometric map $f:M\to M'$ is a diffeomorphism. 
\end{proposition}

	The proposition implies, in particular, that a connected open subset $U$ of a toric dually flat manifold 
	$(M,h,\nabla)$ cannot be toric, unless $U=M$. 

	The proof of Proposition \ref{nfeknkefnknwk} 
	is based on the existence of lifts (see Section \ref{ncekndkrnkenknk}) and on elementary results 
	about local diffeomorphisms and Lie group homomorphisms that we now recall:
	\begin{enumerate}[(i)]
	\item Let $(M,g)$ and $(M',g')$ be connected Riemannian manifolds with $M$ complete, and let $f:M\to M'$ be a local isometry. 
		Then $M'$ is complete and $f$ is a Riemannian covering map (see, e.g., \cite{Lee}, Theorem 6.23).
	\item Let $M$ be a connected manifold. If $M$ is simply connected, then every smooth covering of $M$ is a diffeomorphism
		(see, e.g., \cite{Lee}, Corollary A.59).
	\item Let $G$ and $K$ be Lie groups and $\rho:G\to K$ a Lie group homomorphism. If $\rho$ is injective, 
		then $\rho$ is an immersion. (This follows from the global rank theorem and the fact that a Lie group 
		homomorphism has constant rank.)
	\end{enumerate}
	
\begin{proof}[Proof of Proposition \ref{nfeknkefnknwk}]
	Let $\Phi:\mathbb{T}^{n}\times N\to N$ and $\Phi':\mathbb{T}^{n}\times N'\to N'$ be regular torifications of 
	$M$ and $M'$, respectively. Let $\pi=\kappa\circ \tau:TM\to M$ and $\pi':\kappa'\circ \tau':TM'\to M'$ be toric factorizations (see 
	Section \ref{ncekndkrnkenknk}). Since $N$ and $N'$ are regular, Proposition \ref{ndcknkdnfknnkk} implies that $f$ 
	has a lift, say $\widetilde{f}:N\to N'$. By definition, $\widetilde{f}$ is a K\"{a}hler immersion that is equivariant 
	in the sense that there is a Lie group homomorphism $\rho:\mathbb{T}^{n}\to \mathbb{T}^{n}$ 
	such that $\widetilde{f}\circ \Phi_{a}=\Phi'_{\rho(a)}\circ \widetilde{f}$ for all $a\in \mathbb{T}^{n}$.
	We claim that 
	\begin{enumerate}[(1)] 
	\item $\widetilde{f}$ is a diffeomorphism.
	\item $\rho$ is bijective.
	\end{enumerate}
 	Indeed, since $N$ and $N'$ have the same dimension, $\widetilde{f}$ is a local isometry, and so 
	$\widetilde{f}$ is a smooth covering map by (i). Because $N$ is simply connected, 
	$\widetilde{f}$ is also a diffeomorphism by (ii). This shows (1). To see (2), 
	assume that $\rho(a)=e$ for some $a\in \mathbb{T}^{n}$. Let $x\in N^{\circ}$ be arbitrary.  
	Because $\widetilde{f}$ is equivariant, we have 
	\begin{eqnarray*}
		\widetilde{f}(\Phi_{a}(x))=\Phi'_{\rho(a)}(\widetilde{f}(x))=\Phi'_{e}(\widetilde{f}(x))=\widetilde{f}(x),
	\end{eqnarray*}
	which implies $\Phi_{a}(x)=x$, since $\widetilde{f}$ is a diffeomorphism. Thus 
	$a$ is in the stabilizer of $x$, and since $\Phi$ is free at $x$, $a=e$. This shows that $\rho$ is injective.
	Let us show that $\rho$ is surjective. By (iii), $\rho$ is 
	an immersion between manifolds of the same dimension, and so it is a local diffeomorphism. Thus the image of 
	$\rho$ is open in the connected set $\mathbb{T}^{n}$. It is also closed, because $\mathbb{T}^{n}$ is compact. Therefore 
	$\rho(\mathbb{T}^{n})=\mathbb{T}^{n}$, which shows that $\rho$ is surjective and completes the proof of (2). 
	
	It follows from (1) and (2) that $\widetilde{f}$ is a K\"{a}hler isomorphism which is equivariant 
	relative to the Lie group isomorphism $\rho$. This implies, in particular, that $\widetilde{f}$ restricts 
	to a diffeomorphism from $N^{\circ}$ to $(N')^{\circ}$.
	
	Next we show that $f$ is injective. Let $x$ and $y$ be points in $M$ such that $f(x)=f(y)$. 
	By surjectivity of $\kappa:N^{\circ}\to M$, there are points $x'$ and $y'$ in $N^{\circ}$ such that 
	$x=\kappa(x')$ and $y=\kappa(y')$. To show that $x=y$, it suffices to show that $x'$ and $y'$ are in the 
	same $\mathbb{T}^{n}$-orbit, because $\kappa$ is a 
	principal $\mathbb{T}^{n}$-bundle (see Section \ref{ncekndkrnkenknk}). Using the formula 
	$f\circ \kappa=\kappa'\circ \widetilde{f}$, which holds in $N^{\circ}$, and the fact that $\widetilde{f}$ is an 
	equivariant diffeomorphism, we see that 
	\begin{eqnarray*}
		f(x)=f(y) \,\,\,\,\,\,\,&\Rightarrow& \,\,\,\,\,\,\,f(\kappa(x'))=f(\kappa(y'))\\
					&\Rightarrow& \,\,\,\,\,\,\,(\kappa'\circ \widetilde{f})(x')=
					(\kappa'\circ \widetilde{f})(y')\\
					&\Rightarrow& \,\,\,\,\,\,\,\exists\,a\in \mathbb{T}^{n}\,\,:\,\,\,
						\widetilde{f}(x')=\Phi'_{a}(\widetilde{f}(y'))=\widetilde{f}(\Phi_{\rho^{-1}(a)}(y'))\\
					&\Rightarrow& \,\,\,\,\,\,\,\exists\,a\in \mathbb{T}^{n}\,\,:\,\,\,
						x'=\Phi_{\rho^{-1}(a)}(y')\\
					&\Rightarrow& \,\,\,\,\,\,\,\textup{$x'$ and $y'$ are in the same $\mathbb{T}^{n}$-orbit.}
	\end{eqnarray*}
	Surjectivity of $f$ follows from the formula $\kappa'\circ \widetilde{f}=f\circ \kappa$ and the fact that 
	$\kappa'$ and $\widetilde{f}:N^{\circ}\to (N')^{\circ}$ are surjective.
\end{proof}

\section{One-dimensional toric dually flat spaces}\label{neknkrfnekwnk}

	In this section, we prove the main results of this paper (Theorems \ref{nefknknfeknkn} and \ref{nfeknkfnedkwnk}). 
	Let $(M,h,\nabla)$ be a connected dually flat manifold of dimension 1
	with Hessian scalar curvature $S:M\to \mathbb{R}$. We assume that there is a global 
	coordinate system $x:M\to \mathbb{R}$ which is affine with respect to $\nabla$. Let $I=x(M)\subseteq \mathbb{R}$. 
	We wish to determine conditions under which $(M,h,\nabla)$ is toric.

        In order to simplify the notation, we will denote by $(x,\dot{x})$ the local coordinates on $TM$ associated to 
	$x$ instead of $(q,r)$ (see Section \ref{nkwwnkwnknknfkn}). Thus, 
        $(x,\dot{x})\big(a\tfrac{\partial}{\partial x}\big\vert_{p} \big)=(x(p),a)$, where $p\in M$ and $a\in \mathbb{R}$. 

\begin{lemma}\label{krnkfnkrneknk}
	Suppose that $(M,h,\nabla)$ is toric with regular torification $N$. If the space of K\"{a}hler functions 
	$\mathscr{K}(N)$ separates the points of $N$, then 
	\begin{eqnarray}\label{neknkenfknknk}
		\frac{d}{dx}\big[\Gamma'(x) - \Gamma(x)^{2}\big] =0
	\end{eqnarray}
	for all $x\in I$, where $\Gamma = \Gamma(x)$ is the unique Christoffel symbol associated to $h$ and $\Gamma'(x)=\tfrac{d\Gamma}{dx}(x)$. 
\end{lemma}
\begin{proof}
	Let $f \in \mathscr{K} (TM)$ be arbitrary. By Propositon \ref{c kdlfldfkdl}, $f=f(x,\dot{x})$ satisfies 
	the following equations: 
	\begin{eqnarray}\label{enjknknkefnk}
	\dfrac{\partial^2f}{\partial x^2} - \dfrac{\partial^2f}{\partial \dot{x}^2}  =  
		2\Gamma \dfrac{\partial f}{\partial x} \,\,\,\,\,\,\,\,\textup{and} \,\,\,\,\,\,\,\,
		\dfrac{\partial^2 f}{\partial x \partial\dot{x}}  =  \Gamma \dfrac{\partial f}{\partial \dot{x}}.
	\end{eqnarray}	
	The second equation is equivalent to $\tfrac{\partial h}{\partial x}=\Gamma h$, where 
	$h=h(x, \dot{x})= \frac{\partial f}{\partial \dot{x}},$ and thus 
	$h(x, \dot{x})= e^{p(x)}\lambda(\dot{x})$, where $\lambda, 
	p: \mathbb{R} \rightarrow \mathbb{R}$ are smooth functions and $p'(x)=\Gamma(x)$ for all $x\in I$ 
	(see, e.g., \cite{Ahmad}, Theorem 1.4.2). 
	Integrating the equation $h= \frac{\partial f}{\partial \dot{x}}$ with respect to $\dot{x}$ yields 
 	   \begin{eqnarray}\label{fuckahlerTE}
    		f(x, \dot{x}) = a(\dot{x})e^{p(x)}+b(x),
	    \end{eqnarray}
	where $a$ is a primitive of $\lambda$ on $\mathbb{R}$ and $b: I \rightarrow \mathbb{R}$ is smooth function.
	It follows from \eqref{fuckahlerTE} that 
	\begin{eqnarray*}
    \left\{
         \begin{array}{l}
             \dfrac{\partial f}{\partial x} = a(\dot{x})e^{p(x)}\Gamma(x)+b'(x), \vspace{.8em} \\
		 \dfrac{\partial^2 f}{\partial x^2} = a(\dot{x})e^{p(x)}\Gamma(x)^{2}+ a(\dot{x})e^{p(x)}\Gamma'(x) + b''(x),  \vspace{.2cm}\\
             \dfrac{\partial f}{\partial \dot{x}} = a'(\dot{x})e^{p(x)}, \vspace{.2cm} \\
		     \dfrac{\partial^2 f}{\partial \dot{x}^2} = a''(\dot{x})e^{p(x)}.      
         \end{array}
     \right.
	\end{eqnarray*}
     (Here we use the notation $a'$ and $a''$ to denote the first and second ordinary derivatives of $a$.) 
	Substituting this into the first equation in \eqref{enjknknkefnk} yields 
	  \begin{eqnarray*}
		  a(\dot{x})\big[\Gamma'(x) - \Gamma(x)^{2}\big] - a''(\dot{x}) = e^{-p(x)}[2\Gamma(x)b'(x) - b''(x)].
	  \end{eqnarray*}

     Differentiating with respect to $\dot{x},$ we get

	  \begin{eqnarray*}
		  a'(\dot{x})\big[\Gamma'(x) - \Gamma(x)^{2}\big] - a'''(\dot{x}) =0.
          \end{eqnarray*}

     Differentiating again with respect to $x,$ we obtain

          \begin{eqnarray}\label{eqA2}
		  a'(\dot{x})\frac{d}{dx}\big[\Gamma'(x) - \Gamma(x)^{2}\big] =0.
          \end{eqnarray}        

	We claim that $\frac{d}{dx}[\Gamma'(x) - \Gamma(x)^{2}] = 0$ for all $x\in I$. 
	To see this, assume for the sake of contradiction that there exists some $x_{0} \in I$ such that 
      \begin{eqnarray*}
	      \dfrac{d}{dx}\bigg|_{x=x_0}[\Gamma'(x) - \Gamma(x)^{2}] \neq 0.
      \end{eqnarray*}
	By \eqref{eqA2}, this implies that $a'(\dot{x})=0$ for all $\dot{x} \in I$ and hence 
	$a(\dot{x})\equiv a$ is a constant function. Thus $f$ is independent of $\dot{x}$. 
	It follows from this that the group $\{\tau_{u}\,\,\big\vert\,\,u\in \mathbb{R}\}$ of all 
	translations $\tau_{u}:TM \rightarrow TM$ defined by $ \tau_{u}(x, \dot{x})= (x, \dot{x}+u)$, 
	is a subgroup of $G:=\{ \phi \in \textup{Diff}(T M)\mid f \circ \phi = f, \; \forall f \in \mathscr{K} (TM) \}$. 
	Because $\mathscr{K}(N)$ separates the points of $N$ by hypothesis, 
	the group $G$ coincides with $\Gamma(\mathscr{L})$, where $\mathscr{L}$ is the fundamental lattice 
	(see Proposition \ref{nfeknkwneknk}). Thus 
	$\{\tau_{u}\,\,\big\vert\,\,u\in \mathbb{R}\}\subseteq \Gamma(\mathscr{L})$. It follows from this 
	that $\Gamma(\mathscr{L})$ is uncountable, 
	which is a contradiction. The result follows.
\end{proof}

\begin{lemma}\label{nenkwnkefnwkn}
	Suppose that $(M,h,\nabla)$ is toric with regular torification $N$, and that the space of K\"{a}hler functions 
	$\mathscr{K}(N)$ separates the points of $N$. Then the Hessian scalar curvature $S$ is constant on $M$. 
\end{lemma}
\begin{proof}
	Since $\textup{dim}M=1$, the Christoffel symbol $\Gamma$ associated to $h$ satisfies 
	$\Gamma=\tfrac{1}{2}\tfrac{\partial \ln(h)}{\partial x}$, and thus 
	\begin{eqnarray}\label{nfeknknfekn}
		\Gamma'(x)=\dfrac{1}{2}\dfrac{\partial^{2}\ln(h)}{\partial x^{2}}
	\end{eqnarray}
	for all $x\in I$. By comparing \eqref{nfeknknfekn} with Proposition \ref{nfendwnfeknkdwnkn}, we obtain 
	$S=\tfrac{\Gamma'}{h}$, from which it follows that 
	\begin{eqnarray}\label{nvrkenwknrekwn}
		S'=\dfrac{\partial}{\partial x}\bigg( \dfrac{\Gamma'}{h}\bigg)=\dfrac{\Gamma''h-h'\Gamma'}{h^{2}}. 
	\end{eqnarray}
	By Lemma \ref{krnkfnkrneknk}, $\Gamma''=2\Gamma\Gamma'$, and thus \eqref{nvrkenwknrekwn} can be rewritten as 
	\begin{eqnarray}\label{nvrkenoookwn}
		S' =\dfrac{2\Gamma\Gamma'h-h'\Gamma'}{h^{2}}=\dfrac{2\Gamma'}{h}\bigg(\Gamma -\dfrac{h'}{2h}\bigg). 
	\end{eqnarray}
	Since $\tfrac{h'}{2h}=\tfrac{1}{2}\tfrac{\partial \ln(h)}{\partial x}=\Gamma$, \eqref{nvrkenoookwn} implies that 
	$S'$ is identically zero on $M$. The lemma follows. 
\end{proof}


	Throughout the rest of this section, we assume, in addition to the conditions at the beginning of this section, 
	that $(M,h,\nabla)$ is toric with regular torification $\Phi:\mathbb{T}\times N\to N$, 
	and that the space of K\"{a}hler functions $\mathscr{K}(N)$ separates the points of $N$. 
	By Lemma \ref{nenkwnkefnwkn}, the Hessian sectional curvature $S$ is constant on $M$. Thus there is $\lambda\in \mathbb{R}$ such 
	that $S(p)=\lambda$ for all $p\in M$. 

	Let $h(x)=h(\tfrac{\partial}{\partial x},\tfrac{\partial}{\partial x})$ denote the local expression 
	of $h$. Since $S$ is constant, $h(x)$ is one of the functions in Proposition \ref{nkwknkrnkwdnknkn}. 
	Below we give, in each case, a basis for the space $\mathscr{K}(TM)$ of 
	K\"{a}hler functions on $TM$ (see Proposition \ref{c kdlfldfkdl}). 
	Note that, in any case, $\textup{dim}\,\mathscr{K}(TM)\leq 4$, since $\textup{dim}\,(TM)=2$ (see Section \ref{nfeknkfnekndk}). \\

{\footnotesize	
	\begin{center}
	\renewcommand{\arraystretch}{2.1} 
	\begin{tabular}{|l|l|l|}
	\hline
	$h(x)$         &   $S$    &    Basis for $\mathscr{K}(TM)$   \\
	\hline 
	$e^{ax+b}$     & $\lambda=0$ & $1, \; e^{a x}, \; \cos(a\dot{x}/2)e^{ax/2}, \; \sin(a\dot{x}/2)e^{ax/2}$ \\  [0.8em]
	\hline
	$e^{b}$                   & $\lambda=0$ & $1, \; x, \; \dot{x}, \; \dfrac{x^2 + \dot{x}^2}{2}$ \\ [0.8em]
	\hline
	$\dfrac{a}{\lambda \cos^{2}\big(\sqrt{a}x+b\big)}$	 & $\lambda>0$ &  $1, \; \text{tan}(\sqrt{a}\, x+b), \;
									\dfrac{\cosh(\sqrt{a}\,\dot{x})}{\cos(\sqrt{a}\, x+b)}, \; 
									\dfrac{\sinh(\sqrt{a}\, \dot{x})}{\cos(\sqrt{a}\,x+b)}$ 
								     \\[0.8em]
	\hline
	$\dfrac{a}{\lambda \sinh^{2}\big(\sqrt{a}x+b\big)}$	 & $\lambda>0$ &  $1, \; \text{cotanh}(\sqrt{a}\, x+b), \;
								\dfrac{\cos(\sqrt{a}\, \dot{x} + b)}{\sinh(\sqrt{a}\, x + b)}, \; 
								\dfrac{\sin(\sqrt{a} \dot{x} + b)}{\sinh(\sqrt{a} x + b)}$ \\[0.8em]
	\hline
	$\dfrac{1}{\lambda(x+b)^{2}}$                            & $\lambda>0$ & $1, \; \dfrac{1}{x+b}, \; 
								   \dfrac{\dot{x}}{x+ b}, \; 
								    \dfrac{x^2 + \dot{x}^2}{x + b}$  \\[0.8em]
	\hline 
	$\dfrac{-a}{\lambda \cosh^{2}\big(\sqrt{a}x+b\big)}$    & $\lambda<0$ &  $1, \; \text{tanh}(\sqrt{a}\, x+b), 
		  					            \dfrac{\cos(\sqrt{a}\,\dot{x} + b)}{\cosh(\sqrt{a}\, x + b)}, \; 
								    \dfrac{\sin (\sqrt{a}\, \dot{x} + b)}{\cosh(\sqrt{a}\, x + b)}$  \\[0.8em]
	\hline
	\end{tabular}
	\end{center}
}

	\textbf{}\\

	In the table above, $a$ is a nonzero real number satisfying $a>0$ if $\lambda \neq 0$, and $b\in \mathbb{R}$. Regarding 
	hyperbolic functions, we use the notation $\textup{tanh}(x) = \frac{\sinh(x)}{\cosh(x)}$ and 
	$\textup{cotanh}(x) =\frac{\cosh(x)}{\sinh(x)}$. 
	
	If $h(x)$ is either $e^{b}$ or $\tfrac{a}{\lambda \cos^{2}(\sqrt{a}x+b)}$ or $\tfrac{1}{\lambda(x+b)^{2}}$, then 
	a straightforward verification shows that $\mathscr{K}(TM)$ separates the points of $TM$, which contradicts our 
	assumption that $(M,h,\nabla)$ is toric (see Corollary \ref{nekdnwkwndknk}). Therefore we are left with three possibilities:
	$h(x)$ is either  $e^{ax+b}$ ($\lambda=0$) or $\tfrac{a}{\lambda \sinh^{2}(\sqrt{a}x+b)}$ ($\lambda>0$) or 
	$\tfrac{-a}{\lambda \cosh^{2}\ (\sqrt{a}x+b\ )}$ ($\lambda<0$).

%

	Write $c=-\lambda$, and let $M_{c}$ be the dually flat manifold defined in Section \ref{nceknwkenfknk}. 
	Recall that $M_{c}$ is toric with regular torification $\Phi_{c}:\mathbb{T}\times F(c)\to F(c)$ 
	(see Proposition \ref{nekwnkendkwnkn}). 

	Let $\varphi:M\to M_{c}$ be the map defined by 
	\begin{eqnarray*}
		\varphi(x)= \left\lbrace
			\begin{array}{llll}
				ax+b-\ln(a^{2})          & \textup{if} &  h(x)=e^{ax+b},  & (\lambda=0)  \\[0.5em]
				\varepsilon(\sqrt{a}x+b) & \textup{if} &  h(x)=\tfrac{a}{\lambda \sinh^{2}(\sqrt{a}x+b)}, & (\lambda>0) \\[0.5em]
			        \sqrt{a}x+b              & \textup{if} &  h(x)=\tfrac{-a}{\lambda \cosh^{2}(\sqrt{a}x+b)}, & (\lambda<0) 
			\end{array}
		\right.
	\end{eqnarray*}
	where $x\in I\cong M$, and where $\varepsilon=1$ if $I\subseteq (-\infty,-b/\sqrt{a})$ and $\varepsilon=-1$ 
	if $I\subseteq (-b/\sqrt{a},\infty)$. 
	
	A straightforward verification shows that $\varphi$ is an affine isometric map, and so 
	Proposition \ref{nfeknkefnknwk} implies that 

\begin{lemma}
	The map $\varphi:M\to M_{c}$ is an isomorphism of dually flat manifolds. 
\end{lemma}


	Summarizing, we have proved the following result.

\begin{theorem}\label{nefknknfeknkn}
	Let $(M,h,\nabla)$ be a 1-dimensional connected dually flat manifold, with global 
	affine coordinate system $x:M\to \mathbb{R}$. Suppose that $M$ is toric and that 
	the space of K\"{a}hler functions on the regular torification $N$ separates the points of $N$.
	Then $M\cong M_{c}$ (isomorphism of dually flat manifolds) for some $c\in \mathbb{R}$. 
\end{theorem}

\begin{remark}
	If $M\cong M_{c}$, then obviously $\Phi_{c}:\mathbb{T}\times F(c)\to F(c)$ is a regular torification of $M$. 
\end{remark}

	If, in addition to the conditions in Theorem \ref{nefknknfeknkn}, we assume that $M=\mathcal{E}$ is a 1-dimensional 
	exponential family defined over a finite set, then the constant $c$ must be of the form 
	$\tfrac{1}{p}$, where $p\geq 1$ is an integer. This follows from the following 

\begin{theorem}\label{nfeknkfnedkwnk}
	In the notation of Section \ref{infnekfkenkfnk}, let $\mathcal{E}=\mathcal{E}_{C,F}$ be a 1-dimensional exponential 
	family defined over a finite set $\Omega=\{x_{0},...,x_{m}\}$, and let $\Omega_{red}=\{0,1,...,p\}$ denote the sample space 
	on which the reduced exponential family $\mathcal{E}_{red}$ is defined. Suppose that $\mathcal{E}$ is toric and that the space of K\"{a}hler 
	functions on the regular torification $N$ separates the points of $N$. Then 
	\begin{enumerate}[(1)]
	\item $\mathcal{E}_{red}\sim \mathcal{B}(p)$ (equivalence of exponential families). 
	\item $\mathcal{E}\cong M_{1/p}$ (isomorphism of dually flat manifolds). 
	\item $\mathbb{T}\times \mathbb{P}_{1}(\tfrac{1}{p})\to \mathbb{P}_{1}(\tfrac{1}{p})$, 
		$([t],[z_{0},z_{1}])\mapsto [e^{2i\pi  t}z_{0},z_{1}]$ is a regular torification of $\mathcal{E}$.
	\end{enumerate}
\end{theorem}
\begin{proof}
	This follows from Proposition \ref{nkdnkenkknk}, Lemma \ref{nenkwnkefnwkn}, Remarks \ref{nefkwnkneknk} and \ref{nrvkknkwenken}. 
\end{proof}

\section*{Acknowledgement}
	The first author was supported by the \textit{Coordena\c{c}\~ao de Aperfeiçoamento de Pessoal de N\'ivel Superior -- 
	Brasil (CAPES) -- Finance Code 001}.

\begin{footnotesize}\bibliography{bibtex}\end{footnotesize}
\end{document}